\documentclass[a4paper,12pt,reqno]{amsart}
\usepackage{amscd}
\usepackage{amsmath}
\usepackage{ulem}
\usepackage{amssymb}
\usepackage{amsthm}
\usepackage[usenames,dvipsnames]{color}
\usepackage{color,soul}
\usepackage{enumerate, mdwlist}  
\usepackage{tikz}
\usepackage{url}

\input xy
\xyoption{all}  

\title{ Cotype zeta function for subrings of $Z[t]/(t^3)$}

\begin{document}

%theorem environments
\newtheorem{thm}{Theorem}
\newtheorem{prop}[thm]{Proposition}
\newtheorem{conj}[thm]{Conjecture}
\newtheorem{lem}[thm]{Lemma}
\newtheorem{cor}[thm]{Corollary}
\newtheorem{axiom}[thm]{Axiom}
\newtheorem{sheep}[thm]{Corollary}
\newtheorem{deff}[thm]{Definition}
\newtheorem{fact}[thm]{Fact}
\newtheorem{example}[thm]{Example}
\newtheorem{slogan}[thm]{Slogan}
\newtheorem{remark}[thm]{Remark}
\newtheorem{quest}[thm]{Question}
\newtheorem{zample}[thm]{Example}

\newcommand{\sthat}{\hspace{.1cm}| \hspace{.1cm}}
%shortcuts for curve names
%\newcommand{\p}{\mathbb{P}^1}
%\newcommand{\ga}{\mathbb{G}_a}
%\newcommand{\gm}{\mathbb{G}_m}
%\newcommand{\spec}{ \mathit{Spec} }
\newcommand{\id}{\operatorname{id} }
\newcommand{\acl}{\operatorname{acl}}
\newcommand{\dcl}{\operatorname{dcl}}
\newcommand{\irr}{\operatorname{irr}}
\newcommand{\aut}{\operatorname{Aut}}
\newcommand{\fix}{\operatorname{Fix}}

\newcommand{\oo}{\mathcal{O}}
\newcommand{\aaa}{\mathcal{A}}
\newcommand{\mm}{\mathcal{M}}
\newcommand{\curg}{\mathcal{G}}
\newcommand{\bbf}{\mathbb{F}}
\newcommand{\A}{\mathbb{A}}
\newcommand{\R}{\mathbb{R}}
\newcommand{\Q}{\mathbb{Q}}
\newcommand{\C}{\mathbb{C}}
\newcommand{\cc}{\mathcal{C}}
\newcommand{\dd}{\mathcal{D}}
\newcommand{\N}{\mathbb{N}}
\newcommand{\Z}{\mathbb{Z}}
\newcommand{\cF}{\mathcal F}
\newcommand{\cB}{\mathcal B}
\newcommand{\cU}{\mathcal U}
\newcommand{\cV}{\mathcal V}
\newcommand{\cG}{\mathcal G}
\newcommand{\cD}{\mathcal D}
\newcommand{\curly}{\mathcal{C}}
\newcommand{\durly}{\mathcal{D}}
\newcommand{\fff}{\mathcal{F}}
\newcommand{\calc}{\mathcal{C}}
\newcommand{\GG}{\mathbb{G}}
\newcommand{\PP}{\mathbb{P}}
\newcommand{\Gal}{\mathrm{Gal}}
\newcommand{\Aut}{\mathrm{Aut}}
\newcommand{\signature}{\mathrm{sign}}

\definecolor{mypink3}{cmyk}{0, 0.7808, 0.4429, 0.1412}

\newcommand{\mahrad}[1]{{\color{blue} \sf $\clubsuit\clubsuit\clubsuit$ Mahrad: [#1]}}
\newcommand{\ramin}[1]{{\color{red}\sf $\clubsuit\clubsuit\clubsuit$ Ramin: [#1]}}
\newcommand{\ramintak}[1]{{\color{mypink3}\sf $\clubsuit\clubsuit\clubsuit$ Ramin2: [#1]}}
\newcommand{\ramintakloo}[1]{{\color{green} \sf $\clubsuit\clubsuit\clubsuit$ Ramin3 [#1]}}

\DeclareRobustCommand{\hlgreen}[1]{{\sethlcolor{green}\hl{#1}}}
\DeclareRobustCommand{\hlcyan}[1]{{\sethlcolor{cyan}\hl{#1}}}
\newcommand{\Fmodtor}{F^\times / \mu(F)}

\author{Sarthak Chimni} 
\address{Department of Mathematics, Statistics, and Computer Science, University of Illinois at Chicago, 851 S Morgan St (M/C 249), Chicago, IL 60607}
\email{sarthakchimni@gmail.com}
\author{Ramin Takloo-Bighash}
\email{rtakloo@uic.edu}

\maketitle

\section{Introduction}
The \textit{cotype} of a sublattice  $ \Lambda \subset \Z^n$ is defined as follows. 
By elementary divisor theory, there is a unique $n$ tuple of integers $(\alpha_1, \dots, \alpha_n) = (\alpha_1(\Lambda), \dots, \alpha_n(\Lambda))$ such that the finite abelian group $\Z^n/\Lambda$ is isomorphic to the sum of cyclic groups. 
\[
(\Z/\alpha_1\Z) \oplus (\Z/\alpha_2\Z) \oplus \cdots \oplus (\Z/\alpha_n\Z)
\]
where $\alpha_{i+1}|\alpha_i$ for $1 \leq i \leq n-1$. The $n$-tuple $(\alpha_1(\Lambda), \dots, \alpha_n(\Lambda))$ is called the \textit{cotype} of $\Lambda$. The largest index $i$ for which $\alpha_i \neq 1$ is called the \textit{corank} of $\Lambda$. A sublattice $\Lambda$ of  corank $0$ or $1$ is called {\em cocyclic}. If $\Lambda$ is cocyclic, then  $\Z^n/\Lambda$ is a cyclic group. In \cite{Brakenhoff} Brakenhoff shows that the number of cocyclic subrings of $\Z^n$ of index $p^e$ for any $e > 0$ is $ n \choose 2$.\\
Victor Petrogadsky  \cite{Petrogadsky} introduced the multiple zeta function  
\[
\zeta_{\Z^n}(s_1,s_2, \dots,s_n) = \sum_{|\Z^n:\Lambda| < \infty} {\alpha_1(\Lambda)}^{-s_1}\cdots {\alpha_n(\Lambda)}^{-s_n}
\]
to study the distribution of sublattices $\Lambda$ of $\Z^n$ of different  cotypes.
In \cite{CKK} the authors further study this function and find asymptotics for the distribution of coranks of sublattices of $\Z^n$. Gautam Chinta and Nathan Kaplan had suggested that one can use an analogous multiple zeta function to study the distribution of subrings of $R=\Z^n$ or $\Z[t]/(t^n)$ of different  corank given by
\begin{equation}
\zeta_{R}(s_1,s_2, \dots,s_n) = \sum_{\substack{|R:S| < \infty \\ S  \text{ is a subring } of R }} {\alpha_1(S)}^{-s_1}\cdots {\alpha_{n-1}(S)}^{-s_{n-1}}
\end{equation}
Observe that $\alpha_n$ is always 1 as $S$ is a subring. 

\

On the same lines as Lemma 1.1 of \cite{Petrogadsky} we have an Euler product factorization of $\zeta_{R}(s_1,s_2, \dots,s_{n-1})$ as follows.
\[
\zeta_{R}(s_1,s_2, \dots,s_n) = \prod_p \zeta_{R,p}(s_1,s_2,\cdots,s_{n-1}) 
\]
Where 
\[
\zeta_{R,p}(s_1,s_2,\cdots,s_n) = \sum_{\substack{|R:S| = p^k, k \geq 0 \\ S \text{ subring } of R}} \alpha_1(S)^{-s_1} \alpha_2(S)^{-s_2} \cdots \alpha_{n-1}^{-s_{n-1}}(S)
\]

\

In this note we adapt the methodology of \cite{KMT}, based on \cite{GSS}, to study the above multiple zeta function for the case where $R = \Z[t]/(t^3)$ and obtain explicit results.  We will show 

\begin{equation}
\zeta_{\Z[t]/(t^3),p}(s_1,s_2) = \frac{  1+ X  +pX^2 - pX^3Y -p^2X^4Y-p^2X^5Y  }{(1-XY)(1-pX^2Y)(1-p^2X^3)} 
\end{equation}

with $X = p^{-s_1}, Y = p^{-s_2}$.  This result allows us to isolate the contribution of cocyclic subrings to the total subring zeta function. 
Namely, let $a_n^{cyclic}(\Z[t]/(t^3))$ be the number of subrings of $\Z[t]/(t^3)$ which are cocyclic and of index $n$ in $\Z[t]/(t^3)$. Also let 
$$
\zeta_{\Z[t]/(t^3),p}^{\text{cyclic}}(s) = \sum_{n=1}^\infty \frac{a_n^{cyclic}(\Z[t]/(t^3))}{n^s}
$$
Then it follows that $\zeta_{\Z[t]/(t^3)}^{\text{cyclic}}(s) $ is equal to 
\begin{equation}
\zeta(s)\zeta(2s-1)\zeta(3s-2) \prod_p(
1  - p^{-2s} - p^{1-3s} + p^{1-4s}- p^{2-4s} + p^{2-5s}). 
\end{equation}
Applying a Tauberian theorem gives 
\[
\sum_{n \leq B} a_n^{cyclic}(\Z[t]/(t^3)) \sim  CB(\ln{B})^2, \;  \;   B \to \infty, 
\]
with $C =  \frac{1}{12} \prod_p (1 -3p^{-2} +2p^{-3})$. For comparison, as we will see, the total number of subrings of $\Z[t]/(t^3)$ of index bounded by $B$ grows like 
\[ \frac{1}{12\zeta(2)} B (\ln{B})^2
\]
which is of the same order of magnitude as the number of cocyclic subrings of index bounded by $B$.

\

The second author wishes to thank the Simons Foundation for partial support of his work through a Collaboration Grant. The authors also wish to thank Gautam Chinta and Nathan Kaplan for helpful conversations. This note was inspired by unpublished computation of Chinta where he had treated the case of $\Z^3$. 

\

The paper is organized as follows: In \S \ref{padic} we introduce $p$-adic integration techniques that we will use to compute the multiple subring zeta function for $\Z[t]/(t^3)$. We carry out the main $p$-adic computation in \S \ref{thecase}. The theorems are stated in \S \ref{theorems}.

\section{The $p$-adic integral}\label{padic}
Let $R = \Z[t]/(t^n)$. Note that the additive structure of $R$ is the same as that of $\Z^n$ and that $\beta = \{1,t,\dots,t^{n-1}\}$ is a basis for $R$ as a lattice. We represent elements of $R$ as row vectors of length $n$ via the map that takes $x_j \to e_{n-j}$, where $e_k$ denotes the $k$-th standard basis vector, and is extended by linearity. With this identification, the multiplicative identity 1 is represented by the row vector $(0,0,\dots,1)$. We consider the multiplication of two row vectors $u$, $v$, denoted $u\circ v$, as the row vector representing the product of the corresponding polynomials.

\

The paper \cite{GSS} introduced a $p$-adic formalism to study the local Euler factors $\zeta_{R,p}(s)$. Fix a $\Z$-basis for $R$ and identify $R$ with $\Z^n$ as above. The multiplication in $R$ is given by a bi-additive map
\[
\circ : \Z^n \times \Z^n \to \Z^n
\]
which extends to a bi-additive map
\[
\circ_p : \Z_p^n \times \Z_p^n \to \Z_p^n
\]
giving $R_p = R \otimes_{\Z} \Z_p$ the structure of a $\Z_p$ -algebra. 

\

Let $\mathcal{M}_p(\beta)$ be the subset of the set of $n \times n$ lower triangular matrices $M$ with entries in $\Z_p$ such that if the rows of $M = (x_{ij})$ are denoted by $v_1, \dots v_n$, then for all $i,j$ satisfying $1 \leq i,j \leq n$, there are $p$-adic integers $c_{ij}^1, \dots c_{ij}^n$ such that 
\begin{equation} \label{multiplicativity}
v_i \circ v_j = \sum_{k=1}^n c_{ij}^kv_k 
\end{equation}

Let $dM$ be the normalized additive Haar measure on $T_n(\Z_p)$, the set of $ n \times n$ lower triangular matrices with entries in $\Z_p$. Then proposition 3.1 of \cite{GSS} can be adapted to write $\zeta_{R,p}(s_1, \dots, s_n)$ as a $p$-adic integral.

\begin{multline} \label{p-adic integral}
\zeta_{R,p}(s_1, \dots, s_n) = (1-p^{-1})^{1-n} \int_{\mathcal{M}_p(\beta)} |x_{11}|^{s_1-n+1} |x_{22}|^{s_1-n+2}\cdots \\ |x_{n-1,n-1}|^{s_1-1}|g_{n-1}|^{s_2-s_1}   | g_{n-2}|^{s_3-s_2}\cdots  |g_1|^{s_{n-1} -s_{n-2}} dM
\end{multline}
where $g_k $ is the gcd of the $k \times k$ minors of $M$. The above equation follows from the following proposition.

\begin{prop} \label{cocyclic}
If $S$ is a  sublattice of $\Z^n$ of full rank generated by the rows of a matrix $M$ then we can write $R/S \cong (\Z/\alpha_1\Z) \oplus (\Z/\alpha_2\Z) \oplus \cdots \oplus(\Z/\alpha_n\Z)$ where $\alpha_{i+1} \mid \alpha_i$ and $\alpha_{n-k+1}\alpha_{n-k+2} \cdots \alpha_n$ is equal to the gcd of all $k \times k$ minors of $M$, with the convention that if all $k \times k$ minors are $0$, then their gcd is $0$.
\end{prop}

\section{The case where $R = \Z[t]/(t^3)$ }\label{thecase}
Let $ R = \Z[t]/(t^3)$. Then the corresponding multiple subring zeta function is given by
\[
\zeta_{\Z[t]/(t^3)}(s_1,s_2) = \sum_{\substack{|R/S| < \infty \\ S \text{ subring } of R}} \alpha_1(S)^{-s_1} \alpha_2(S)^{-s_2}
\]
where $(\alpha_1(S), \alpha_2(S))$ is the cotype of the subring $S$. It follows from lemma 1.1 in \cite{Petrogadsky} that there is an Euler Product for $\zeta_{\Z[t]/(t^3)}(s_1,s_2)$ given by
\[ 
\zeta_{\Z[t]/(t^3)}(s_1,s_2) = \prod_p \zeta_{\Z[t]/(t^3),p} (s_1,s_2) 
\]
where
\[
\zeta_{\Z[t]/(t^3),p}(s_1,s_2) = \sum_{\substack{|R:S| = p^k, k \geq 0 \\ S \text{ subring of }  \Z[t]/(t^3)}} \alpha_1(S)^{-s_1} \alpha_2(S)^{-s_2}
\]
The goal of this section is to prove the following theorem: 

\begin{thm}\label{main-thm}
We have 
\begin{equation}
\zeta_{\Z[t]/(t^3),p}(s_1,s_2) = \frac{  1+ X  +pX^2 - pX^3Y -p^2X^4Y-p^2X^5Y  }{(1-XY)(1-pX^2Y)(1-p^2X^3)} 
\end{equation}
with $X = p^{-s_1}, Y = p^{-s_2}$. 
\end{thm}

We start with a lemma. 

\begin{lem}
For $\Z[t]/(t^3), \mathcal{M}_p(\beta)$ is the subset of $ 3 \times 3 $ lower triangular matrices M with entries in $\Z_p$ 
\[
M = 
\begin{bmatrix}
x   &   0    &   0 \\
y   &    z   &   0 \\ 
0    &   0   &   1
\end{bmatrix}
\]
where $|z|^2 \leq |x|$.
\end{lem}

\begin{proof}
 The multiplicativity condition from equation (\ref{multiplicativity}) gives $(y,z,0)\circ(y,z,0) = (yt^2 + zt)^2 = z^2t^2 = c_1xt^2 + c_2(yt^2 + zt)$. So we get $c_1 = \frac{z^2}{x}$ and $c_2=0$. Since $c_1 \in \Z_p$ we have $\lvert\frac{z^2}{x}\rvert \leq 1$.
\end{proof}

\begin{prop} \label{p-adic integral 3}. We have 
\begin{multline} 
\zeta_{\Z[t]/(t^3),p}(s_1,s_2) = {(1 - p^{-1})}^{-2} \int_{|z|^2 \leq |x| \leq 1} {|x|}^{s_1-2} {|z|}^{s_1-1} \\ {\text{max} \{|x|,|y|,|z|\}}^{s_2-s_1} dx dy dz. 
\end{multline}
\end{prop}
\begin{proof}
This follows directly from applying (\ref{p-adic integral}) to $\Z[t]/(t^3)$.
\end{proof}

\

The remainder of this section is devoted to the explicit computation of this integral.  We use the following two formulae extensively in the computation.  
For $ Re(s) > -1$, let $I(s) = \int_{\Z_p} {|z|}^sdz$. Then it is very easy to see that 

$$I(s) = \frac{1-p^{-1}}{1-p^{-s-1}}.$$
For $\alpha \in \Z_p$ and $Re(s) > -1$ define
\begin{equation*}
I_0(s,\alpha) = \int_{|\alpha| < |z| \leq 1} {|z|}^sdz 
\end{equation*}
and 
\begin{equation*}
I_1(s, \alpha) = \int_{|\alpha| \leq |z| \leq 1} {|z|}^sdz . 
\end{equation*}
Then 
$$
I_0(s, \alpha) = I(s)(1 - {|\alpha|}^{s+1})
$$
and 
$$
 I_1(s, \alpha) = I(s)(1 - p^{-s-1}{|\alpha|}^{s+1}). 
$$

\

We now compute the $p$-adic integral in Proposition \ref{p-adic integral 3}. We divide our computation into six cases depending on the valuation of $x$, $y$ and $z$ and then sum up their contributions at the end. Note that in the final expressions for the $J_i, x= p^{-s_1}$ and $ y = p^{-s_2}$.
\\\\
$\textbf{CASE 1: } |z| \leq |y| \leq |x|$ 

\

Here

\[
J_1= {(1 - p^{-1})}^{-2} \int_{|z| \leq |y| \leq |x| \leq 1} {|x|}^{s_2-2} {|z|}^{s_1-1}dx dy dz. 
\]
Change of variables, $z$ to $zx$ and $y$ to $yx$ gives 
\[
=  \frac{1}{(1-p^{-1})^2} \int_{|z| \leq |y| \leq 1} |x|^{s_1+s_2-1}|z|^{s_1-1}dx dy dz
\]
Integrating with respect to $x$
\[
= \frac{1}{(1-p^{-1})(1-p^{-s_1-s_2})}\int_{|z| \leq |y| \leq 1} |z|^{s_1-1}dy dz
\]
Integrating with respect to $y$
\[
= \frac{1}{(1-p^{-1})(1-p^{-s_1-s_2})}\int_{|z| \leq 1} (1-p^{-1}|z|)|z|^{s_1-1}dz
\]
Integrating with respect to $z$
\[
=\frac{1}{(1-xy)}\bigg(\frac{1}{1-x} - \frac{p^{-1}}{(1-p^{-1}x)}\bigg). 
\]
As a result, 
\begin{equation}
J_1=\frac{1-p^{-1}}{(1-x)(1-p^{-1}x)(1-xy)}
\end{equation}
\\
$\textbf{CASE 2: } |y| < |z| \leq |x|$ 

\

Then 
\[
J_2 = {(1 - p^{-1})}^{-2} \int_{|y| < |z| \leq |x| \leq 1 } {|x|}^{s_2-2} {|z|}^{s_1-1}dx dy dz
\]
Change of variables, $z$ to $zx$ and $y$ to $yx$ gives 
\[
= {(1 - p^{-1})}^{-2} \int_{|y| < |z| \leq 1 }  {|x|}^{s_1+s_2-1}{|z|}^{s_1-1}dx dy dz
\]
Integrating with respect to $x$
\[
\frac{(1 - p^{-1})^{-1}}{1-p^{-s_1-s_2}}  \int_{|y| < |z| \leq 1 } {|z|}^{s_1-1}dy dz
\]
Integrating with respect to $z$
\[
\frac{1}{(1-p^{-s_1-s_2})(1-p^{-s_1})} \int_{|y| < 1} (1- |y|^{s_1}) dy
\]
Change of variables $y$ to $py$
\[
=\frac{p^{-1}}{(1-p^{-s_1-s_2})(1-p^{-s_1})} \int_{\Z_p} (1- |py|^{s_1}) dy
\]
\[
= \frac{p^{-1}}{(1-xy)(1-x)}\bigg(1 - \frac{x(1-p^{-1})}{1-p^{-1}x}\bigg). 
\]
So
\begin{equation}
J_2= \frac{p^{-1}}{(1-xy)(1-p^{-1}x)}. 
\end{equation}

\

$\textbf{CASE 3: } |x| \leq |z| < |y|$ 

\

Then 
\[
J_3 = \Z_p(s_1,s_2) = {(1 - p^{-1})}^{-2} \int_{|z|^2 \leq |x| \leq |z| < |y| \leq 1 } {|x|}^{s_1-2} {|z|}^{s_1-1}|y|^{s_2-s_1}dx dy dz
\]
Integrate with respect to $x$
\[
= \frac{1}{(1-p^{-1})(1-p^{-s_1+1})}\int_{|z| < |y| \leq 1}(|z|^{s_1-1} - p^{-s_1+1}|z|^{2s_1-2}) |z|^{s_1-1} |y|^{s_2-s_1} dy dz
\]
Integrating with respect to $y$
\[
\frac{1}{(1-p^{s_1-s_2-1})(1-p^{-s_1+1})}\int_{|z|< 1}(|z|^{2s_1-2} - p^{-s_1+1}|z|^{3s_1-3})(1-|z|^{s_2-s_1+1}) dz
\]
Change of variables $z$ to $pz$
\begin{multline*}
=\frac{p^{-1}}{(1-p^{s_1-s_2-1})(1-p^{-s_1+1})}\bigg(\int_{|z| \leq 1} p^{-2s_1+2}|z|^{2s_1-2} - p^{-4s_1+4}|z|^{3s_1-3} \\ - p^{-s_1-s_2+1}|z|^{s_1+s_2-1}+p^{-3s_1-s_2+3}|z|^{2s_1+s_2-2} dz\bigg)
\end{multline*}
Integrating with respect to $z$
\[
=\frac{p^{-1}x(1-p^{-1})}{(1-px)(x-p^{-1}y)}\bigg(\frac{p^2x^2}{1-px^2} - \frac{p^4x^4}{1-p^2x^3} - \frac{pxy}{1-xy} + \frac{p^3x^3y}{1-px^2y}\bigg). 
\]
\begin{equation}
J_3= \frac{(1-p)x^2(p^2x^3+px^2y-px-1)}{(1-px^2)(1-xy)(1-p^2x^3)(1-px^2y)}. 
\end{equation}

\

$\textbf{CASE 4: } |z| < |x| < |y|$ 

\

Then 
\[
J_4 = {(1 - p^{-1})}^{-2} \int_{|z| < |x| < |y| \leq 1 } {|x|}^{s_1-2} {|z|}^{s_1-1}|y|^{s_2-s_1}dx dy dz
\]
Integrating with respect to $y$
\[
= \frac{(1-p^{-1})^{-1}}{1-p^{s_1-s_2-1}} \int_{|z| < |x| < 1 } {|x|}^{s_1-2} {|z|}^{s_1-1}(1-|x|^{s_2-s_1+1})dx dz
\]
Change of variables $x$ to $px$ and $z$ to $pz$
\[
= \frac{p^{-2s_1+1}(1-p^{-1})^{-1}}{1-p^{s_1-s_2-1}} \int_{|z| < |x| \leq 1 } {|x|}^{s_1-2} {|z|}^{s_1-1}(1-p^{s_1-s_2-1}|x|^{s_2-s_1+1})dx dz
\]
Integrating with respect to $x$
\begin{multline*}
=\frac{p^{-2s_1+1}}{(1-p^{s_1-s_2-1})(1-p^{-s_1+1})} \int_{|z| < 1 } |z|^{s_1-1}(1-|z|^{s_1-1})  \\ -\frac{p^{-s_1-s_2}}{(1-p^{s_1-s_2-1})(1-p^{-s_2})} \int_{|z| < 1 } |z|^{s_1-1}(1-|z|^{s_2}) dz
\end{multline*}
Change of variable $z$ to $pz$
\begin{multline*}
\frac{p^{-2s_1}}{(1-p^{s_1-s_2-1})(1-p^{-s_1+1})} \int_{|z| \leq 1 } p^{-s_1+1} |z|^{s_1-1} - p^{-2s_1+2} |z|^{2s_1-2} dz \\ - \bigg(\frac{p^{-s_1-s_2-1}}{1-p^{s_1-s_2-1})(1-p^{-s_2})} \int_{|z| \leq 1 }  p^{-s_1+1} |z|^{s_1-1} - p^{-s_1-s_2+1}|z|^{s_1+s_2-1} dz \bigg)
\end{multline*}
Integrating with respect to $z$
\[
=\frac{x^3(1-p^{-1})}{(x-p^{-1}y)(1-px)}\bigg(\frac{px}{1-x} - \frac{p^2x^2}{1-px^2}\bigg) - \frac{p^{-1}(1-p^{-1})x^2y}{(x-p^{-1}y)(1-y)}\bigg(\frac{px}{1-x}-\frac{pxy}{1-xy}\bigg)
\]
\[
\frac{(p-1)x^4}{(x-p^{-1}y)(1-x)(1-px^2)} - \frac{(1-p^{-1})x^3y}{(x-p^{-1}y)(1-x)(1-xy)}. 
\]
This gives 
\begin{equation}
J_4=\frac{(p-1)x^3}{(1-x)(1-px^2)(1-xy))}. 
\end{equation}

\

$\textbf{CASE 5: } |x| < |y| \leq |z|$ 

\

Then 
\[
J_5 = {(1 - p^{-1})}^{-2} \int_{|z|^2 \leq |x| < |y| \leq |z| < 1 } {|x|}^{s_1-2} {|z|}^{s_2-1}dx dy dz
\]
Integrating with respect to $y$
\[
= {(1 - p^{-1})}^{-2} \int_{|z|^2 \leq |x|  <|z| < 1 } {|x|}^{s_1-2} {|z|}^{s_2-1}(|z|-|x|)dxdz
\]
\[
=  {(1 - p^{-1})}^{-2} \int_{|z|^2 \leq |x| < |z| < 1 } {|x|}^{s_1-2} {|z|}^{s_2} - |x|^{s_1-1}|z|^{s_2-1} dx dz
\]
\[
=I_1 - I_2
\]
with 
\[
I_1 = {(1 - p^{-1})}^{-2} \int_{|z|^2 \leq |x| < |z| < 1 } {|x|}^{s_1-2} {|z|}^{s_2}dx dz
\]
and 
\[
I_2 = {(1 - p^{-1})}^{-2} \int_{|z|^2 \leq |x| < |z| < 1 }  |x|^{s_1-1}|z|^{s_2-1} dx dz. 
\]
We first compute $I_1$. Integrating with respect to $x$
\[
I_1=\frac{p^{-s_1+1}}{(1-p^{-1})(1-p^{-s_1+1})}\int_{|z|<1} (|z|^{s_1-1}-|z|^{2s_1-2})|z|^{s_2} dz 
\]
Change of variable $z$ to $pz$
\[
I_1 = \frac{p^{-s_1+s_2}}{(1-p^{-1})(1-p^{-s_1+1})}\int_{|z| \leq1} p^{-s_1+1}|z|^{s_1+s_2-1}- p^{-2s_1+2}|z|^{2s_1+s_2-2}dz
\]
Integrating with respect to $z$
\[
= \frac{p^{-s_1+s_2}}{1-p^{-s_1+1}}\bigg(\frac{p^{-s_1+1}}{1-p^{-s_1-s_2}} - \frac{p^{-2s_1+2}}{1-p^{-2s_1-s_2+1}}\bigg)
\]
\[
I_1 = \frac{px^2y}{(1-xy)(1-px^2y)}.
\]
Next we compute $I_2$.  Integrating with respect to $x$
\[
I_2 = \frac{p^{-s_1}}{(1 - p^{-1})(1-p^{-s_1})} \int_{|z| < 1 } (|z|^{s_1}-|z|^{2s_1})|z|^{s_2-1}dz
\]
Change of variable $z$  to $pz$
\[
I_2 =  \frac{p^{-s_1-s_2}}{(1 - p^{-1})(1-p^{-s_1})} \int_{|z| \leq 1 } (p^{-s_1}|z|^{s_1}-p^{-2s_1}|z|^{2s_1})|z|^{s_2-1}dz
\]
Integrating with respect to $z$
\[
I_2 = \frac{p^{-s_1-s_2}}{1 - p^{-s_1}}\bigg(\frac{p^{-s_1}}{1-p^{-s_1-s_2}} - \frac{p^{-2s_1}}{1-p^{-2s_1-s_2}}\bigg)
\]
\[
I_2 = \frac{x^2y}{(1-xy)(1-x^2y)}
\]
Putting $I_1$ and $I_2$ together 
\begin{equation}
J_5 = I_1 - I_2 =  \frac{(p-1)x^2y}{(1-xy)(1-x^2y)(1-px^2y)}. 
\end{equation}
\\
$\textbf{CASE 6: } |y| \leq |x| < |z|$ 

\

$\textbf{SUBCASE 1: } |z|^2 \leq |y|$
\[
{J_{6}}^1 = {(1 - p^{-1})}^{-2} \int_{|z|^2 \leq |y| \leq |x| < |z| \leq 1 } {|x|}^{s_1-2} {|z|}^{s_2-1}dx dy dz
\]
Change of variables $x$ to $xz$ and $y$ to $yz$
\[
= (1 - p^{-1})^{-2} \int_{|z| \leq |y| \leq |x| <  1 } {|x|}^{s_1-2} {|z|}^{s_1+s_2-1}dx dy dz
\]
Change of variables $x$ to $px$, $y$ to $py$ and $z$ to $pz$
\[
p^{-2s_1-s_2}(1 - p^{-1})^{-2}\int_{|z| \leq |y| \leq |x| \leq  1 }{|x|}^{s_1-2} {|z|}^{s_1+s_2-1}dx dy dz
\]
Integrating with respect to $x$
\[
\frac{p^{-2s_1-s_2}}{(1 - p^{-1})(1-p^{-s_1+1})}\int_{|z| \leq |y| \leq 1}(1-p^{-s_1+1}|y|^{s_1-1})|z|^{s_1+s_2-1} dy dz
\]
Integrating with respect to $y$
\begin{multline*}
=\frac{p^{-2s_1-s_2}}{(1-p^{-1})(1-p^{-s_1+1})}\int_{|z| \leq 1}(1-p^{-1}|z|)|z|^{s_1+s_2-1} dz \\ - \bigg(\frac{p^{-4s_1-s_2+1}}{(1-p^{-s_1+1})(1-p^{-s_1})}\int_{|z| \leq 1}(|z|^{s_1+s_2-1} - p^{-s_1}|z|^{2s_1+s_2-1} dz \bigg)
\end{multline*}
Integrating with respect to $z$
\[
= \frac{x^2y}{1-px}\bigg(\frac{1}{1-xy} - \frac{p^{-1}}{1-p^{-1}xy}\bigg) - \frac{px^3y(1-p^{-1})}{(1-px)(1-x)}\bigg(\frac{1}{1-xy}-\frac{x}{1-x^2y}\bigg)
\]
\[
=\frac{x^2y(1-p^{-1})}{(1-px)(1-p^{-1}xy)(1-xy)} - \frac{px^3y(1-p^{-1})}{(1-px)(1-xy)(1-x^2y)}.
\]
Hence, 
\[
{J_{6}}^1=\frac{x^2y(1-p^{-1})}{(1-xy)(1-x^2y)(1-p^{-1}xy)}
\]

\

$\textbf{SUBCASE 2 } |y| < |z|^2$

\[
{J_{6}}^2 = {(1 - p^{-1})}^{-2} \int_{|y| < |z|^2 \leq |x| < |z| \leq 1 } {|x|}^{s_1-2} {|z|}^{s_2-1}dx dy dz
\]
Change of variables $x$ to $xz$ and $y$ to $yz$
\[
= (1 - p^{-1})^{-2} \int_{|y| < |z| \leq |x| < 1 } {|x|}^{s_1-2} {|z|}^{s_1+s_2-1}dx dy dz
\]
Change of variables from $x$ to $px$, $y$ to $py$ and $z$ to $pz$
\[
p^{-2s_1-s_2}(1 - p^{-1})^{-2} \int_{|y| < |z| \leq |x| \leq 1 } {|x|}^{s_1-2} {|z|}^{s_1+s_2-1}dx dy dz
\]
Integrating with respect to $x$
\[
\frac{p^{-2s_1-s_2}}{(1-p^{-1})(1-p^{-s_1+1})}\int_{|y| < |z| \leq 1}|z|^{s_1+s_2-1}-p^{-s_1+1}|z|^{2s_1+s_2-2} dz
\]
Integrating with respect to $z$
\begin{multline*}
=\frac{p^{-2s_1-s_2}}{(1-p^{-s_1-s_2})(1-p^{-s_1+1})}\int_{|y| < 1} (1-|y|^{s_1+s_2}) dy \\ - \bigg(\frac{p^{-3s_1-s_2+1}}{(1-p^{-s_1+1})(1-p^{-2s_1-s_2+1})}\int_{|y| < 1} (1 - |y|^{2s_1+s_2-1})dy\bigg)
\end{multline*}
Change of variable $y$ to $py$
\begin{multline*}
=\frac{p^{-2s_1-s_2-1}}{(1-p^{-s_1-s_2})(1-p^{-s_1+1})}\int_{|y| \leq 1} (1-p^{-s_1-s_2}|y|^{s_1+s_2}) dy \\ - \bigg(\frac{p^{-3s_1-s_2}}{(1-p^{-s_1+1})(1-p^{-2s_1-s_2+1})}\int_{|y| \leq1} (1 - p^{-2s_1-s_2+1}|y|^{2s_1+s_2-1})dy\bigg)
\end{multline*}
\[
= \frac{p^{-1}x^2y}{(1-xy)(1-px)}\bigg(1 - \frac{xy(1-p^{-1})}{1-p^{-1}xy}\bigg) - \frac{x^3y}{(1-px)(1-px^2y)}\bigg(1- \frac{px^2y(1-p^{-1})}{1-x^2y}\bigg). 
\]
Summing up, 
\[
{J_{6}}^2=\frac{p^{-1}x^2y}{(1-p^{-1}xy)(1-x^2y)}. 
\]
Adding the contribution from the two subcases we get  
\begin{equation}
J_6 = {J_{6}}^1 + {J_{6}}^2 = \frac{x^2y}{(1-x^2y)(1-xy)}. 
\end{equation}

We can now prove Theorem \ref{main-thm}. 

\begin{proof}[Proof of Theorem \ref{main-thm}]
We have 
$$
\zeta_{\Z[t]/(t^3,p}(s_1,s_2) = \sum_{i=1}^{6} J_i. 
$$
The proof now follows from an elementary computation. 
\end{proof}

\section{Corollaries} \label{theorems}
In this section we draw some corollaries from the Theorem \ref{main-thm}.  Our first application is to compute the subring zeta function of $R = \Z[t]/(t^3)$, $\zeta_{\Z[t]/(t^3)}(s)$. 
\begin{cor}
We have 
\begin{equation}
 \zeta_{\Z[t]/(t^3)}(s) =  \frac{\zeta(s)\zeta(2s-1)\zeta(3s-2)}{\zeta(4s-2)}. 
\end{equation}
\end{cor}
\begin{proof} We do this locally.  Setting  $s_2 = s_1 = s$ in Theorem \ref{main-thm}, with $Z = p^{-s}$, we get  
\[
\zeta_{\Z[t]/(t^3),p}(s) =  \frac{      1+ Z +pZ^2- pZ^4-p^2Z^5-p^2Z^6}{(1-Z^2)(1-pZ^3)(1-p^2Z^3)} 
\]
\begin{equation*}
= \frac{1+pZ^2}{(1-Z)(1-p^2Z^3)}. 
\end{equation*}
Next, 
\begin{equation}
 \zeta_{\Z[t]/(t^3)}(s) = \prod_p \zeta_{\Z[t]/(t^3),p}(s) = \frac{\zeta(s)\zeta(2s-1)\zeta(3s-2)}{\zeta(4s-2)}
\end{equation}
\end{proof}
An application of a Tauberian theorem gives the following corollary: 
\begin{cor}
Let $a_n(\Z[t]/(t^3))$ be the number of subrings with an identity of $\Z[t]/(t^3)$ of index $n$. Then 
\[
\sum_{n \leq B} a_n(\Z[t]/(t^3)) \sim \frac{1}{12\zeta(2)} B (\ln{B})^2, \;  \;   B \to \infty
\]

\end{cor}

We now examine cocyclic subrings.  Let $a_n^{cyclic}(\Z[t]/(t^3))$ be the number of subrings of $\Z[t]/(t^3)$ which are cocyclic and of index $n$ in $\Z[t]/(t^3)$. Also let 
$$
\zeta_{\Z[t]/(t^3),p}^{\text{cyclic}}(s) = \sum_{n=1}^\infty \frac{a_n^{cyclic}(\Z[t]/(t^3))}{n^s}. 
$$

\begin{cor} The function $\zeta_{\Z[t]/(t^3),p}^{\text{cyclic}}(s)$ is equal to 
\begin{equation}
\zeta(s)\zeta(2s-1)\zeta(3s-2) \prod_p(
1  - p^{-2s} - p^{1-3s} + p^{1-4s}- p^{2-4s} + p^{2-5s}). 
\end{equation}
\end{cor}
\begin{proof} Again the proof is local. To compute the $p$-local factor of the zeta function we set $Y = 0 $, i.e., $s_2 \to \infty$, in Theorem \ref{main-thm} to obtain 
$$
\frac{1 + p^{-s}+ p^{1-2s}}{1-p^{2-3s}}. 
$$
The result is now obvious. 
\end{proof}

Since $\zeta_{\Z[t]/(t^3),p}^{\text{cyclic}}(s)$ has a pole of order three at $s=1$, a Tauberian theorem gives the following corollary: 

\begin{cor} We have 
\[
\sum_{n \leq B} a_n^{cyclic}(\Z[t]/(t^3)) \sim  CB\ln{B}^2, \;  \;   B \to \infty
\]
where $C =  \frac{1}{12} \prod_p (1 -3p^{-2} +2p^{-3})$.
\end{cor}

So that the ratio of cyclic subrings to total subrings of $\Z[t]/(t^3)$ is given by
\[
\zeta(2) \prod_p (1 -3p^{-2} +2p^{-3}). 
\]

\end{document}